\documentclass{amsart}

\usepackage{latexsym}
\usepackage{euscript}
\usepackage{amsfonts}
\usepackage{amsmath}

\usepackage{color}

\usepackage{amssymb}

\usepackage{enumerate}

 \newtheorem{thm}{Theorem}[section]
 
 \newtheorem{lem}[thm]{Lemma}
 \newtheorem{prop}[thm]{Proposition}
 \theoremstyle{definition}
 
 \theoremstyle{remark}

 \numberwithin{equation}{section}

\begin{document}
%
%
\title[Dilations of holomorphic semigroups]{Kre\u{\i}n space unitary dilations of Hilbert space holomorphic semigroups}

\author{S.A.M.~Marcantognini}
\address{Department of Mathematics,
Instituto Venezolano de Investigaciones Cient\'{\i}ficas,
Km. 11 Carretera Panamericana,
Altos de Pipe, Edo. Miranda,
Venezuela;
\newline
CONICET-Instituto Argentino de Matem\'aticas ``Alberto P. Calder\'on'', Saavedra 15, Piso 3, 1083, C.A.B.A., Argentina;
\newline
Instituto de Ciencias, Universidad Nacional de General Sarmiento, Juan Mar\'{\i}a Gutierrez, 1613, Los Polvorines, Pcia. de Buenos Aires, Argentina.}
\email{stefania.marcantognini@gmail.com}
\thanks{The work was done during my sojourn in the Instituto Argentino de Matem\'aticas in a position funded by CONICET. I am deeply grateful to the institute for the hospitality and to CONICET for the financial support.}




\date{ }
\dedicatory{ }

\begin{abstract}
The infinitesimal generator $A$ of a strongly continuous semigroup on a Hilbert space is assumed to satisfy that $B_\beta:=A-\beta$ is a sectorial operator of angle less than $\frac{\pi}{2}$ for some $\beta \geq 0$. If $B_\beta$ is dissipative in some equivalent scalar product then the Naimark-Arocena Representation Theorem is applied to obtain a Kre\u{\i}n space unitary dilation of the semigroup. 
\end{abstract}

\maketitle

\section{Introduction}
\label{sec:introduction}

A Kre\u{\i}n space unitary dilation of a  strongly continuous semigroup $\{T(t)\}$ on a Hilbert space $\mathfrak H$ was built up by B. McEnnis under the assumption that the numerical range of the infinitesimal generator $A$ lies on a sector of  semi-angle $0 < \theta < \frac{\pi}{2}$ around the real line \cite{McE}. McEnnis' construction follows the one given by C.~Davis for uniformly continuous semigroups \cite{Da}. Both constructions lean on the existence of a selfadjoint operator $G$ such that
$$
\frac{d}{dt}\|T(t)h\|^2 = \langle GT(t)h, T(t)h\rangle, \;  
\frac{d}{dt}\|T(t)^*h\|^2 = \langle GT(t)^*h, T(t)^*h\rangle
$$
for all $t > 0$ and all $h \in \mathfrak H$. The key element in McEnnis' is the convexity of the numerical range (the Haussdorff-Toeplitz Theorem) which grants the $m$-$\theta$-dissipativeness of $A-\beta$ for some $\beta \geq 0$ and, in consequence, a contractive holomorphic extension of $\{e^{-\beta t} T(t)\}$ within the sector $|{\arg}(z)| < \frac{\pi}{2} - \theta$. 

The sectoriality of the numerical range can be replaced by sectoriality of the spectrum together with some norm estimates of the resolvent. The latter conditions allow a functional calculus which, in turn, gives a one-to-one correspondence between the type of closed operators $A$ satisfying those constraints and the bounded holomorphic strongly continuous semigroups $\{T(z)\}$ (see, for instance, \cite{Ha}). On the other hand, if the calculus is $H^\infty$-bounded then the operator  $A$ is dissipative in some equivalent Hilbert space inner product (refer to \cite{Ha} for details). In sum, McEnnis' result can be achieved under the weaker conditions that combine sectoriality (of the spectrum) with bounded $H^\infty$-calculus. We apply the Naimark-Arocena Representation Theorem to obtain the Kre\u{\i}n space unitary dilation, though. So this note serves as a slightly more general result than the one by McEnnis as well as an alternate of its proof. 

Apart from the present section, which functions as a brief introduction, this note comprises two other sections:  Section~2 includes some preliminaries while Section~3 presents the results.   

\section{Preliminaries}\label{sec:prelims}
In the sequel we assume that all Hilbert spaces are complex and separable. 
 
Given a Hilbert space  
 $(\mathfrak H,\langle\cdot,\cdot\rangle)$, we denote by $\mathfrak L(\mathfrak H)$ the linear space of all bounded linear operators on $\mathfrak H$. When $A$ is a linear operator which is not everywhere defined on $\mathfrak H$, we write $\mathcal D(A)$ for its domain. The symbol $\mathcal R(A)$ stands for the range of $A$. The {\it{numerical range}} of $A$ is the set 
$$\nu(A) := \{ \langle Ax, x\rangle~|~x\in \mathcal D(A), \|x\| =1\}.$$
If $A: \mathcal D(A) \to \mathfrak H$ is a closed linear operator then 
 $$
 \rho(A) := \{\lambda \in \mathbb C~|~(\lambda - A)^{-1} \in \mathfrak L(\mathfrak H)\}, \quad \sigma(A) := \mathbb C \setminus \rho(A)
 $$
and
$$
 R(\lambda, A) := (\lambda -A)^{-1}\quad (\lambda \in \rho(A))
$$
are its resolvent set, spectrum and resolvent, respectively.

\subsection{Holomorphic semigroups}\label{Kato}

References on semigroup theory abound; \cite{Pa} is amongst the classical textbooks on the subject. We recall that a {\it{strongly continuous (one-parameter) semigroup}} on a Hilbert space $(\mathfrak H,\langle\cdot,\cdot\rangle)$ is a family $\{T(t)\} \subseteq \mathfrak L(\mathfrak H)$ parameterized by $t\geq 0$ that satisfies the following conditions:
 \begin{itemize}
 \item[(i)] $T(0) = 1$ and $T(s+t)= T(s)T(t)$ for all $s,t \geq 0$. 
 \item[(ii)] $T(t)$ converges strongly to $1$ as $t\to 0^+$.
 \end{itemize}
 
 When $\|T(t)\| \leq 1$ for all $t\geq 0$, $\{T(t)\}$ is said to be a {\it{contraction semigroup}}. 
  
 The {\it{infinitesimal generator}} $A$ of a strongly continuous semigroup $\{T(t)\}$ is defined as 
 $Ax := \lim_{t\to0^+} t^{-1}(T(t)-1)x$ being the domain $\mathcal D(A)$ of $A$ the set of those $x \in \mathfrak H$ for which the limit exists. $A$ is known to be a densely defined closed operator.

We also recall that a linear operator $A$ on a Hilbert space $\mathfrak H$ with domain $\mathcal D(A)$ is said to be {\it{dissipative}} if ${\mbox{Re}}\langle Ax,x\rangle \leq 0$ for all $x\in \mathcal D(A)$, that is, if $\nu(A)$ is contained in the closed left half-plane. A dissipative operator $A$ is called $m$-{\it{dissipative}}  if  $A$ is closed and $\mathcal R(1-A) =  \mathfrak H$.

\begin{prop}\label{p1} 
The following assertions are equivalent: 
\begin{itemize}
\item[(a)] $A$ is $m$-dissipative. 
\item[(b)] $A$ is a densely defined closed operator and $A^*$ is $m$-dissipative. 
\item[(c)] $A$ is a densely defined closed operator, $\sigma(A) \subseteq \{\lambda \in \mathbb C : {\mbox{\textnormal{Re}}}\,\lambda \leq 0\}$ and $\|R(\lambda,A)\| \leq ( {\mbox{\textnormal{Re}}}\,\lambda)^{-1}$ for all ${\mbox{\textnormal{Re}}}\,\lambda >0$.
\item[(d)] $A$ generates a strongly continuous contraction semigroup.
\end{itemize}
\end{prop}

(a) $\Leftrightarrow$ (d)  and (c) $\Leftrightarrow$ (d) are the renowned Lumer-Phillips Theorem and Hille-Yosida Theorem, respectively. 

Let $0 < \omega \leq \frac{\pi}{2}$. An $\mathfrak L(\mathfrak H)$-valued function $T$ on the additive semigroup 
$$X(\omega):= \{z\in \mathbb C : z\ne 0,\;|{\arg}(z)| < \omega\} \cup \{0\}$$ 
is called a {\it{strongly continuous holomorphic semigroup}} if: 
\begin{itemize}
\item[(i)] $z \mapsto T(z)$ is holomorphic in $X(\omega)$.
\item[(ii)] $T(0) = 1$ and $T(z+z') = T(z)T(z')$ for all $z,z' \in X(\omega)$.
\item[(iii)] For each $0<\phi<\omega$, $T(z)$ converges strongly to $1$ as $z\to 0$ within $X(\phi)$.
\end{itemize}
The holomorphic semigroup $\{T(z)\}$ on $X(\omega)$ is said to be {\it{bounded}} when:
\begin{itemize}
\item[(iv)] $\sup\{\|T(z)\| : z \in X(\phi)\} < \infty$ for each $0<\phi<\omega$.
\end{itemize}

Above and in the following, ${\arg}(z)$ is to be set in the interval $(-\pi, \pi]$ for $z \in \mathbb C \setminus \{0\}$.
  
Let $0 < \theta < \frac{\pi}{2}$. A linear operator $A$ on $\mathfrak H$ with domain $\mathcal D(A)$ is said to be {\it{$\theta$-dissipative}} if 
$\nu(A)$ lies in the sector 
$$S_{0, \theta} := \{\lambda \in \mathbb C : \lambda\ne 0,\;|{\arg}(\lambda)| \geq \pi - \theta\} \cup \{0\}.$$ 
A $\theta$-dissipative operator $A$ is $m$-$\theta$-{\it{dissipative}} if $A$ is $m$-dissipative. 
\begin{prop}\label{p2}
The following assertions are equivalent:
\begin{itemize}
\item[(a)] $A$ is $m$-$\theta$-dissipative.
\item[(b)] $A$ generates a strongly continuous holomorphic semigroup $\{T(z)\}$ on\linebreak  $X(\frac{\pi}{2} - \theta)$ satisfying $\|T(z)\| \leq 1$ for all $z \in X(\frac{\pi}{2} - \theta)$.
\end{itemize}
\end{prop}

We point out that $A$ is said to be ($\theta$ -) {\it{accretive}} if $-A$ is ($\theta$ -) dissipative. The ($m$ -) $\theta$-accretive operators were extensively studied by T.~Kato \cite{Ka}. He called them ($m$ -) {\it{sectorial operators}} but nowadays the term ``{{sectorial}}" is referred to a different type of operator. In the following we will call $A$  a {\it{sectorial operator}} if $\sigma(A) \subseteq  S_{0,\theta}$ for some $0 < \theta < \pi$ and, for each $\theta < \phi < \pi$, $\sup\{
\| \lambda R(\lambda,A)\|~:~\lambda \in \mathbb C \setminus S_{0,\phi}\} < \infty$. We will restrict ourselves to sectorial operators of angle $0< \theta < \frac{\pi}{2}$. 
A comprehensive account of sectorial operators is \cite{Ha}. Therein the definition of sectorial operator is different from the one we adopt, though. Roughly, our corresponds with $-A$ being sectorial in the above mentioned monograph.

The most famous result on the numerical range is the Toeplitz-Hausdorff Theorem  which asserts that the numerical range of any (perhaps unbounded and not densely defined) linear operator on a complex or real  (pre-)Hilbert space  is convex. The convexity of the numerical range is one of the key elements in showing the following result from \cite{Ka} (cf. \cite{McE}).

\begin{prop}\label{beta} Let  $A$ be the infinitesimal generator of a strongly continuous semigroup $\{T(t)\} \subseteq \mathfrak L(\mathfrak H)$. If $\nu(A)$ lies in the sector
$$S_{\alpha, \theta} := \{\lambda \in \mathbb C : \lambda \ne\alpha,\;|{\arg}(\lambda -\alpha)| \geq \pi - \theta\} \cup \{\alpha\}$$
of vertex $\alpha \geq 0$ and semi-angle $0 < \theta < \frac{\pi}{2}$, then there exists $\beta \geq 0$ such that 
$A- \beta$ is $m$-$\theta$-dissipative.
\end{prop}

By combining Proposition \ref{beta} with Proposition \ref{p2} we get that, under the assumption that 
\begin{equation}\label{quasi} 
\nu(A) \subseteq S_{\alpha, \theta} \quad (\alpha \geq 0, 0 < \theta < \pi/2),
\end{equation}
there exists $\beta \geq 0$ such that  
$\{e^{-\beta z}T(z)\}$ is a contraction holomorphic semigroup on $X(\frac{\pi}{2} - \theta)$. Hence, if we consider only real $t$, we obtain that
\begin{equation}\label{bound}
\|T(t)\| \leq e^{\beta t}\quad \mbox{for all } t \geq 0
\end{equation}
and
\begin{equation}\label{ddt}
A^nT(t) = \frac{d}{dt^n}T(t) \quad \mbox{for every } t > 0 \mbox{ and } n\in \mathbb N.
\end{equation}

 Holomorphic semigroups are somehow in between the general class of strongly continuous semigroups and the particular class of uniformly continuous semigroups (for which the infinitesimal generator is bounded). The functional calculus for sectorial operators gives a one-to-one correspondence between sectorial operators $A$ with $0\leq \theta < \frac{\pi}{2}$ and bounded holomorphic strongly continuous semigroups $\{T(z)\}$ on 
 $X(\frac{\pi}{2} - \theta)$ (see \cite{Ha}). Therefore (\ref{ddt}) is granted under the weaker condition of sectoriality. On the other hand, (\ref{bound}) holds if $A - \beta$ is $m$-dissipative.

We replace (\ref{quasi}) by the following hypothesis: there exist $\beta \geq 0$ such that $A-\beta$ is sectorial of angle $0 < \theta < \frac{\pi}{2}$ and ${\mbox{{\mbox{Re}}}}\langle Ax,x\rangle_0 \leq \beta\|x\|_0^2$ for all $x\in \mathcal D(A)$, with $\langle\cdot,\cdot\rangle_0$  an equivalent scalar product on $\mathfrak H$. This happens, for instance, if $A-\beta$ is sectorial and has bounded $H^\infty$-calculus with $H^\infty$-angle less than $\frac{\pi}{2}$ (the reader is referred to \cite{Ha} for the definitions). We remark that it is not always the case that a sectorial operator $A$ on a Hilbert space, with sectorial angle less than 
$\frac{\pi}{2}$, has bounded $H^\infty$-calculus.

\subsection{Kre\u{\i}n spaces}
As familiarity with operator theory on Kre\u{\i}n spaces is presumed, only some notation is introduced. We emphasize that the common Hilbert space notation is carried over into the Kre\u{\i}n space setting.

Given a  fundamental decomposition $\mathfrak {K} = \mathfrak {K}^+ \oplus \mathfrak {K}^-$ of the  Kre\u{\i}n space $(\mathfrak {K}, \langle\cdot,\cdot\rangle)$, we write $|\mathfrak {K}|$ for $\mathfrak {K}$ viewed as  the Hilbert space relative to the fundamental decomposition. Therefore, if $J$ is the corresponding fundamental symmetry, that is, 
$Jx=x^+ - x^-$ whenever $x= x^+  + x^-$ with $x^\pm \in 
\mathfrak {K}^\pm$, then $\langle x, y\rangle_{|\mathfrak {K}|} = \langle Jx  ,y\rangle$ and $\langle x  ,y\rangle = \langle Jx, y\rangle_{|\mathfrak {K}|}$ for all $x, y \in \mathfrak {K}$.

By $\mathfrak {L}(\mathfrak {K})$ we mean the space of all everywhere defined 
continuous linear operators on the Kre\u{\i}n space $\mathfrak K$. The space 
$\mathfrak {L}(\mathfrak {K})$ has the structure of a Banach space depending on the choice of a fundamental decomposition and the associated Hilbert space 
$|\mathfrak K|$. The corresponding {{operator norm}} for $\mathfrak {L}(\mathfrak {K})$ is the norm $\|\cdot\|$ of $\mathfrak {L}(|\mathfrak {K}|)$. We point up that any two operator norms for $\mathfrak {L}(\mathfrak {K})$ are equivalent and provide its topology. If $\mathfrak K$ and $\mathfrak G$ are two Kre\u{\i}n spaces, $\mathfrak {L}(\mathfrak {K}, \mathfrak G)$ and the operator norm are defined likewise.

For each $A \in \mathfrak {L}(\mathfrak {K}, \mathfrak G)$ there is a unique $A^* \in \mathfrak {L}(\mathfrak G,\mathfrak {K})$ so that $\langle Ax, y\rangle_\mathfrak G = \langle x, A^*y\rangle_\mathfrak K$ for all $x \in \mathfrak K$ and $y \in \mathfrak G$. 

We say that (i) $P \in \mathfrak {L}(\mathfrak {K})$ is a projection if $P^2=P=P^*$; (ii)  $U \in \mathfrak {L}(\mathfrak {K})$ is a unitary operator if $U^*U = 1 = UU^*$.

The regular subspaces of $\mathfrak K$ are those that are the ranges of projections. If  $\mathfrak F$ is a regular subspace of  $\mathfrak K$ we write $P_{\mathfrak F}$ to indicate the {\it{orthogonal projection}} from  $\mathfrak K$ onto  $\mathfrak F$.

Standard references on Kre\u{\i}n spaces and operators on them are \cite{A}, \cite{AI} and \cite{B}. We also refer to \cite{DR} and \cite{DR1} as authoritative accounts of the subject.

\subsection{Naimark-Arocena Representation Theorem}
\label{Arocena}

The Naimark Theorem characterizes those operator-valued Toeplitz kernels which have Hilbert space unitary representations. The theorem has an extension to the Kre\u{\i}n space setting due to R.~Arocena \cite{Ar}.
\begin{thm}  \label{rep}
Let $\Gamma$ be a group and denote by $e$ its neutral element. Let $(\mathfrak H, \langle \cdot,\cdot\rangle)$ be a Hilbert space. Let $f : \Gamma \to \mathfrak L(\mathfrak H)$ be a Hermitian function  such that $f(e) = 1$.   Assume that there exists a kernel $k : \Gamma \times \Gamma \to \mathfrak L(\mathfrak H)$ with the following properties:
\begin{itemize}
\item[(i)] $k(e,e) = 1$.
\item[(ii)] $k$ is a majorant of $f$, in the sense that there exists $r>0$ such that
$$\left|\sum\limits_{s,t \in \Gamma} \langle f(t^{-1}s) h(s), h(t)\rangle\right| \leq r \sum\limits_{s,t \in \Gamma}
\langle k(s,t)h(s), h(t)\rangle$$
for every function $h:\Gamma \to \mathfrak H$ with finite support.
\item[(iii)] There exists $R>0$ such that, for any $h:\Gamma \to \mathfrak H$ with finite support such that 
$\sum_{s,t \in \Gamma}\langle k(s,t)h(s), h(t)\rangle >0$, there exists another function $h^\prime:\Gamma \to \mathfrak H$ with finite support satisfying $\sum_{s,t \in \Gamma}\langle k(s,t)h^\prime(s), h^\prime(t)\rangle >0$ and
$$\frac{\left|\sum\limits_{s,t \in \Gamma} \langle f(t^{-1}s) h(s), h^\prime(t)\rangle\right|}
{\left(\sum\limits_{s,t \in \Gamma}\langle k(s,t)h(s), h(t)\rangle\right)^{\frac{1}{2}}
\left(\sum\limits_{s,t \in \Gamma}\langle k(s,t)h^\prime(s), h^\prime(t)\rangle\right)^{\frac{1}{2}}} \geq R.$$
\item[(iv)] There exists a function $\rho : \Gamma \to (0,\infty)$ such that
$$\sum\limits_{s,t \in \Gamma}\langle k(\xi s,\xi t)h(s), h(t)\rangle
\leq \rho(\xi) \sum\limits_{s,t \in \Gamma}\langle k(s,t)h(s), h(t)\rangle$$
for all $h:\Gamma \to \mathfrak H$ with finite support and all $\xi \in \Gamma$.
\end{itemize}
Then there exist a Kre\u{\i}n space $(\mathfrak K, \langle \cdot,\cdot\rangle_\mathfrak K)$ containing $(\mathfrak H, \langle\cdot,\cdot\rangle)$ as regular subspace and a unitary representation $U(s)$ of $\Gamma$ in $\mathfrak K$ such that:
\begin{enumerate}
\item[1)] $f(s) = P_\mathfrak H U(s)|_\mathfrak H$ for all $s \in \Gamma$, and
\item[2)] $\mathfrak K = \bigvee \{U(s)\mathfrak H~:~s\in \Gamma\}$, that is, $\mathfrak K$ is the space generated by $U(s)\mathfrak H$, $s\in \Gamma$.
\end{enumerate}

Conversely, if there exist a Kre\u{\i}n space $(\mathfrak K, \langle \cdot,\cdot\rangle_\mathfrak K)$ containing $(\mathfrak H, \langle\cdot,\cdot\rangle)$ as regular subspace and a unitary representation $U(s)$ of $\Gamma$ in $\mathfrak K$ such that 1) and 2) hold, then there exists a kernel $k$ satisfying (i), (ii), (iii) and (iv).  
\end{thm}

\section{The Dilation}
\label{sec:Dilation}

Hereafter $\{T(t)\}$ is a strongly continuous semigroup of bounded linear operators on a Hilbert space $(\mathfrak H, \langle\cdot,\cdot\rangle)$ with infinitesimal generator $A$. 

Let $\langle \cdot,\cdot \rangle_A$ be the graph scalar product on $\mathcal D(A)$:
$$\langle x, y \rangle_A := \langle x, y\rangle + \langle Ax, Ay \rangle \quad (x,y \in \mathcal D(A)).$$
By the Riesz-Fr\'echet {\mbox{Re}}presentation Theorem, there exists a bounded linear operator $R$ on 
$(\mathcal D(A), \langle \cdot,\cdot \rangle_A)$ such that 
$$\langle Ax,y \rangle + \langle x,Ay \rangle = \langle x,Ry \rangle_A =  \langle Rx,y \rangle_A\quad {\mbox{for all }} x,y \in \mathcal D(A).$$
Since $A$ is a densely defined closed operator on $(\mathfrak H, \langle\cdot,\cdot\rangle)$, the von Neumann Theorem (see \cite{Yo}) assures that $A^*A, A A^*$ are selfadjoint and $1+ A^*A, 1 + A A^*$ are invertible. Let $h \in \mathfrak H$ and write $h = (1+ A^*A)y$, $y \in \mathcal D(A)$. For all $x \in \mathcal D(A)$,
$$\begin{array}{rcl}
\langle Rx, h \rangle &\!\!=\!\!& \langle Rx, (1+A^*A)y \rangle = \langle (1+A^*A)Rx, y \rangle = \langle Rx, y \rangle_A
= \langle x, Ry\rangle_A \\
&\!\!=\!\!& \langle x, (1+A^*A)Ry\rangle = \langle x, (1+A^*A)R(1+A^*A)^{-1}h\rangle.
\end{array}
$$
If $(S, \mathcal D(S))$ is the operator on $(\mathfrak H, \langle\cdot,\cdot\rangle)$ given by $\mathcal D(S) := \mathcal D(A)$ and $Sx:= Rx$, then it follows that $\mathcal D(S^*) = \mathfrak H$ and $S^* = (1+A^*A)S(1+A^*A)^{-1}$. Thus $((1+A^*A)S)^* \supseteq S^*(1+A^*A) = (1+A^*A)S$.

We then get that $H:= (1+A^*A)S$ is a symmetric operator on $(\mathfrak H, \langle\cdot,\cdot\rangle)$ with $\mathcal D(H) = \mathcal D(A)$ such that
\begin{equation}\label{H}
\langle Hx, y\rangle = \langle Ax,y \rangle + \langle x,Ay \rangle \quad {\mbox{for all }} x,y \in \mathcal D(A).
\end{equation}

\begin{lem}\label{lemma1}
If there exists $\beta \geq 0$ such that $A - \beta$ is sectorial of angle $0~<~\theta~<~\frac{\pi}{2}$ and dissipative then the symmetric operator $H$ in (\ref{H}) satisfies 
$$ \langle Hx, x \rangle \leq 2\beta \|x\|^2 \quad {\mbox{for all }} x\in \mathcal D(A)$$
and
$$\frac{d}{dt}\|T(t)h\|^2 = \langle HT(t)h, T(t)h\rangle \quad {\mbox{for all }} t > 0 {\mbox{ and }} h \in \mathfrak H.$$
\end{lem}
\begin{proof}
Since $B:= A-\beta$ is dissipative, it comes that, for all $x \in \mathcal D(A)$, $\langle Hx,x\rangle = 2{\mbox{{\mbox{Re}}}}\langle Bx,x \rangle + 2\beta \|x\|^2 \leq 2\beta \|x\|^2$. On the other hand, $B$ is $m$-dissipative, for $B$ is sectorial of angle $0 < \theta < \frac{\pi}{2}$ and dissipative. Therefore, $B$ generates a strongly continuous bounded holomorphic semigroup $\{S(z)\}$ on $X(\frac{\pi}{2} - \theta)$ such that $\|S(t)\| \leq 1$ in the real semi-axis $t\geq 0$. As $T(t) = e^{\beta t}S(t)$ for all $t\geq 0$, the result follows. \end{proof}

The arguments applied to $A$ can be reproduced on the infinitesimal generator $A^*$ of the adjoint semigroup $\{T(t)^*\}$ to obtain a symmetric operator $H_*$ on $(\mathfrak H, \langle\cdot,\cdot\rangle)$ with $\mathcal D(H_*) = \mathcal D(A^*)$ such that
\begin{equation}\label{H*}
\langle H_*u, v\rangle = \langle A^*u,v \rangle + \langle u,A^*v \rangle \quad {\mbox{for all }} u, v\in \mathcal D(A^*).
\end{equation}

\begin{lem}\label{lemma2}
Under the hypotheses of Lemma \ref{lemma1} one has that the symmetric operator $H_*$ in (\ref{H*}) satisfies 
$$\langle H_*u, u \rangle \leq 2\beta \|u\|^2 \quad {\mbox{for all }} u\in \mathcal D(A^*)$$
and
$$\frac{d}{dt}\|T(t)^*h\|^2 = \langle H_*T(t)^*h, T(t)^*h\rangle \quad {\mbox{for all }} t > 0 {\mbox{ and }} h \in \mathfrak H.$$
\end{lem}
\begin{proof}
Notice that $B^*$ ($= (A-\beta)^*$) is sectorial of the same angle $0 < \theta < \frac{\pi}{2}$ as $B$ \cite{Ha}. Besides, $B^*$ is $m$-dissipative, according with Proposition \ref{p1}.
\end{proof}

We get two densely defined symmetric sesquilinear forms bounded from above with upper bound $2\beta$ by setting
$$
x, y \in \mathcal D(A) = \mathcal D(H) \mapsto \langle Hx, y\rangle \quad{\mbox{and}}\quad u, v \in\mathcal D(A^*) = \mathcal D(H_*) \mapsto
 \langle H_*u, v\rangle.
 $$
Both forms are closable \cite{Ka}. Their corresponding closures
 $\mathfrak g[\cdot,\cdot] $ and 
$\mathfrak g_*[\cdot,\cdot]$ are densely defined closed symmetric sesquilinear forms  such that $\mathfrak g[x,x] \leq 2\beta\|x\|^2$ for all $x\in \mathcal D(\mathfrak g)$ and  $\mathfrak g_*[u,u] \leq 2\beta\|u\|^2$ for all $u\in \mathcal D(\mathfrak g_*)$ \cite{Ka}. The Friedrichs Theorem for symmetric forms (cf. \cite{Ka}) gives two selfadjoint operators $G, G_*$ which are bounded from above by $2\beta$ and  satisfy
\begin{equation}\label{DDT}
\begin{array}{l}
\frac{d}{dt}\|T(t)h\|^2 = \langle GT(t)h, T(t)h\rangle, \;  
\frac{d}{dt}\|T(t)^*h\|^2 = \langle G_*T(t)^*h, T(t)^*h\rangle\\[.15cm]
\mbox{for all } t > 0 \mbox{ and } h \in \mathfrak H.
\end{array}
\end{equation}

A relevant fact established in \cite{Ka} yields $G=G_*$. 

Consider the polar decomposition $G = J|G|$, where $J$ is a selfadjoint partial isometry and $|G|$ is selfadjoint and nonnegative with $\mathcal D(|G|) = \mathcal D(G)$. Then write 
$$Gx = G^+x + G^-x, \quad |G|x = G^+x - G^-x\quad (x \in \mathcal D(G)=\mathcal D(|G|))$$
where $G^\pm \supseteq \frac{1}{2}(G \pm |G|)= \frac{1}{2}(J \pm 1)|G|$.

\begin{lem}\label{GGG*G*}
The following assertions hold true:
\begin{itemize}
\item[(a)] Every $x \in \mathcal D(G)$ satisfies $Jx \in \mathcal D(G)$, $GJx = JGx = |G|x$ and $|G|Jx = J|G|x = Gx$.
\item[(b)] $\left|\langle Gx, x\rangle\right| \leq \langle|G|x, x\rangle$ for every $x \in \mathcal D(G)$.
\item[(c)] For all $x \in \mathcal D(G)$, $\langle G^-x, x\rangle \leq 0$ 
and $0 \leq \langle G^+x, x\rangle \leq 2\beta \|x\|^2$.
\end{itemize}
\end{lem}
The proof is omitted. 

Now we are ready to apply the Naimark-Arocena Representation Theorem to obtain a unitary dilation of $\{T(t)\}$. First we suppose that the infinitesimal generator $A$ is as in Lemma \ref{lemma1}. To comply with the hypotheses of the theorem, take $\Gamma = \mathbb R$ and define $f : \mathbb R \to \mathfrak L(\mathfrak H)$ by
$$
f(s) := \left\{
\begin{array}{ll}
T(s)\quad & \mbox{if } s\geq 0\\
T(-s)^*\quad & \mbox{if } s <  0
\end{array}\right..
$$
It follows that $f$ is Hermitian and $f(0) = 1$. Hence the kernel $(s,t) \mapsto f(s-t)$ is an operator-valued Toeplitz kernel on $\mathbb R$. 

Let $h: \mathbb R \to \mathfrak H$ be a function with finite support. Write $\mbox{supp}(h) \cup \{0\} = \{\sigma_j\}_{j=0}^Q \cup \{\tau_k\}_{k=0}^P$ with 
$\sigma_Q < \cdots < \sigma_1 < \sigma_0 = 0 = \tau_0 < \tau_1 < \cdots < \tau_P$. Define $z (= z_h)$ and 
$y (= y_h)$ by
$$
z(s) := \sum_{t\leq s} T(s-t)^*h(t)\quad(s\leq 0)\quad\mbox{and}\quad y(s):= \sum_{t\geq s} T(t-s)h(t)\quad(s\geq 0).
$$
Then 
\begin{equation}\label{h}
h(s) = \left\{
\begin{array}{ll}
z(s) & \mbox{ if } s \leq \sigma_Q\\ 
z(s) - T(s -\sigma_j)^*z(\sigma_j) &  \mbox{ if } \sigma_j<  s \leq \sigma_{j-1}, \; 1\leq j \leq Q\\ 
y(s) - T(\tau_k- s)y(\tau_k ) &  \mbox{ if } \tau_{k-1} \leq s < \tau_k,\; 1 \leq k\leq P\\
y(s) &  \mbox{ if }  s \geq \tau_P
\end{array}\right..
\end{equation}
In particular, 
$$
z(0) - T(-\sigma_1)^*z(\sigma_1) = h(0) = y(0) - T(\tau_1)y(\tau_1)
$$
whence
$$
z(0) + T(\tau_1)y(\tau_1) = y(0) +  T(-\sigma_1)^*z(\sigma_1).
$$
If $v  (=  v_h)$ is given by 
\begin{equation}\label{v}
v(s) := \left\{
\begin{array}{ll} z(s)\quad &\mbox{if } s < 0\\
z(0) + T(\tau_1)y(\tau_1) = y(0) +  T(-\sigma_1)^*z(\sigma_1)\quad &\mbox{if } s = 0\\
y(s)\quad &\mbox{if } s > 0
\end{array}\right.
\end{equation}
then, via the map $h \mapsto v$, we get that
$$
\begin{array}{l}
\sum\limits_{s,t \in \mathbb R} \langle f(s-t)h(s),h(t)\rangle \\[.15cm]
\quad= \sum\limits_{j=1}^Q \langle (1-T(\sigma_{j-1}-\sigma_j)
T(\sigma_{j-1}-\sigma_j)^*)v(\sigma_j), v(\sigma_j)\rangle 
+  \|v(0)\|^2 \\[.15cm]
\quad\quad+ \sum\limits_{k=1}^P \langle (1-T(\tau_k-\tau_{k-1})^*
T(\tau_k-\tau_{k-1}))v(\tau_k), v(\tau_k)\rangle.
\end{array}
$$
Therefore, by (\ref{DDT}),
$$
\begin{array}{l}
\sum\limits_{s,t \in \mathbb R} \langle f(s-t)h(s),h(t)\rangle \\[.15cm]
\quad= \sum\limits_{j=1}^Q \displaystyle{\int_{\sigma_j}^{\sigma_{j-1}}} \langle (-G)T(u-\sigma_j)^*
v(\sigma_j), T(u-\sigma_j)^*v(\sigma_j)\rangle \,du
+  \|v(0)\|^2 \\[.15cm]
\quad\quad+ \sum\limits_{k=1}^P \displaystyle{\int_{\tau_{k-1}}^{\tau_k}}\langle (-G)T(\tau_k-u)
v(\tau_k), T(\tau_k-u)v(\tau_k)\rangle.
\end{array}
$$
So, from Lemma \ref{GGG*G*} (b),
$$
\begin{array}{l}
\left|\sum\limits_{s,t \in \mathbb R} \langle f(s-t)h(s),h(t)\rangle\right| \\[.15cm]
\quad \leq \sum\limits_{j=1}^Q \displaystyle{\int_{\sigma_j}^{\sigma_{j-1}}} \langle |G|T(u-\sigma_j)^*
v(\sigma_j), T(u-\sigma_j)^*v(\sigma_j)\rangle \,du
+  \|v(0)\|^2 \\[.15cm]
\quad\quad+ \sum\limits_{k=1}^P \displaystyle{\int_{\tau_{k-1}}^{\tau_k}}\langle |G|T(\tau_k-u)
v(\tau_k), T(\tau_k-u)v(\tau_k)\rangle \\[.15cm]
\quad = \sum\limits_{j=1}^Q \displaystyle{\int_{\sigma_j}^{\sigma_{j-1}}} \| CT(u-\sigma_j)^*
v(\sigma_j)\|^2 \,du
+  \|v(0)\|^2 \\[.15cm]
\quad\quad+ \sum\limits_{k=1}^P \displaystyle{\int_{\tau_{k-1}}^{\tau_k}}\| CT(\tau_k-u)
v(\tau_k)\|^2 \,du,
\end{array}
$$
with $C := |G|^{\frac{1}{2}}$. 
Straightforward computations give
$$
\begin{array}{l}
 \sum\limits_{j=1}^Q \displaystyle{\int_{\sigma_j}^{\sigma_{j-1}}} \| CT(u-\sigma_j)^*v(\sigma_j)\|^2 \,du 
\\[.15cm]
\quad = \sum\limits_{s<0}\sum\limits_{t<0} \left\langle\left(\int_{s \vee t}^0 T(u-t)|G|T(u-s)^* \,du\right)h(s),h(t)\right\rangle,
\end{array}
$$
$$
\begin{array}{l}
\|v(0)\|^2 = \sum\limits_{s<0}\sum\limits_{t<0}\langle T(-t)T(-s)^* h(s), h(t)\rangle +
 \sum\limits_{s<0}\sum\limits_{t\geq 0}\langle T(t-s)^* h(s), h(t)\rangle\\[.15cm]
 \quad\quad\quad\quad\quad + \sum\limits_{s\geq 0}\sum\limits_{t<0}\langle T(s-t) h(s), h(t)\rangle +
\sum\limits_{s\geq 0}\sum\limits_{t\geq 0}\langle T(t)^*T(s) h(s), h(t)\rangle
 \end{array}
 $$
and
$$
\begin{array}{l}
\sum\limits_{k=1}^P \displaystyle{\int_{\tau_{k-1}}^{\tau_k}}\| CT(\tau_k-u) v(\tau_k)\|^2 \,du
\\[.15cm]
\quad = \sum\limits_{s > 0}\sum\limits_{t > 0} \left\langle\left(\int_0^{s \wedge t} T(t-u)^*|G|T(s-u) \,du\right)h(s),h(t)\right\rangle.
\end{array}
$$
As a result, by setting 
$$
k(s,t) := \left\{\begin{array}{ll}
\int_{s \vee t}^0 T(u-t)|G|T(u-s)^* \,du + T(-t)T(-s)^* & s,t <0 \\[.15cm]
T(t-s)^*  & s<0, t\geq 0\\[.15cm]
T(s-t) & s \geq 0, t<0\\[.15cm]
T(t)^*T(s) + \int_0^{s \wedge t} T(t-u)^*|G|T(s-u) \,du & s \geq 0, t\geq 0
\end{array}\right.,
$$
we get an operator-valued kernel on $\mathbb R$ such that $k(0,0) = 1$ and
$$
\left|\sum\limits_{s,t \in \mathbb R} \langle f(s-t)h(s),h(t)\rangle\right| \leq 
\sum\limits_{s,t \in \mathbb R} \langle k(s,t)h(s),h(t)\rangle
$$ 
for every $h: \mathbb R \to \mathfrak H$ with finite support. 

It is worth noticing that, for $h: \mathbb R \to \mathfrak H$ with finite support and the corresponding 
$v: \mathbb R \to \mathfrak H$ as defined in (\ref{v}), we have that:
\begin{itemize}
\item[(1)] If $\sigma_j \leq u < \sigma_{j-1}$, $1\leq j \leq Q$, then
$$T(u-\sigma_j)^*v(\sigma_j) = \sum\limits_{s\leq u} T(u-s)^*h(s) = v(u).$$
\item[(2)] If $\tau_{k-1} < u \leq \tau_k$, $1\leq k \leq P$, then
$$T(\tau_k - u)v(\tau_k) = \sum\limits_{s\geq u} T(s-u)h(s) = v(u).$$
\end{itemize}
Here, as before,  
$\sigma_Q < \cdots < \sigma_1 < \sigma_0 = 0 = \tau_0 < \tau_1 < \cdots < \tau_P$ are the real points in 
$\mbox{supp}(h) \cup \{0\}$. Accordingly,
$$\begin{array}{lll}\sum\limits_{s,t \in \mathbb R} \langle f(s-t)h(s),h(t)\rangle & = &\sum\limits_{j=1}^Q 
\displaystyle{\int_{\sigma_j}^{\sigma_{j-1}}} \langle (-G)v(u), v(u)\rangle \,du
+  \|v(0)\|^2\\[.15cm]
&  & \quad + \sum\limits_{k=1}^P \displaystyle{\int_{\tau_{k-1}}^{\tau_k}}\langle (-G)v(u), v(u)\rangle\,du\end{array}$$
and
$$\begin{array}{lll}\sum\limits_{s,t \in \mathbb R} \langle k(s,t)h(s),h(t)\rangle & = &\sum\limits_{j=1}^Q 
\displaystyle{\int_{\sigma_j}^{\sigma_{j-1}}} \langle |G|v(u), v(u)\rangle \,du
+  \|v(0)\|^2\\[.15cm]
&  & \quad + \sum\limits_{k=1}^P \displaystyle{\int_{\tau_{k-1}}^{\tau_k}}\langle |G|v(u), v(u)\rangle\,du.\end{array}$$

If $\sum\limits_{s,t \in \mathbb R} \langle k(s,t) h(s), h(t)\rangle > 0$, set
$$
v^\prime(u) := \left\{
\begin{array}{ll} J v(u)\quad &\mbox{if  either } u < 0 \mbox{ or } u>0\\
v(0)\quad &\mbox{if } u = 0
\end{array}\right..
$$
A function $h^\prime: \mathbb R \to \mathfrak H$ with $\mbox{supp}(h^\prime) = \mbox{supp}(h)$ is uniquely determined by $v^\prime$ from (\ref{h}). Matter-of-factly, $h^\prime$ is given by
$$
h^\prime(s) = \left\{
\begin{array}{ll}
v^\prime(\sigma_Q) & \mbox{ if } s = \sigma_Q\\ 
v^\prime(\sigma_j) - T(\sigma_j -\sigma_{j+1})^*v^\prime(\sigma_{j+1}) &  \mbox{ if } s = \sigma_j, \; 1\leq j \leq Q-1\\ 
v^\prime(0) - T(-\sigma_1)^*v^\prime(\sigma_1) - T(\tau_1)v^\prime(\tau_1) &  \mbox{ if } s = 0\\
v^\prime(\tau_k) - T(\tau_{k+1}- \tau_k)v^\prime(\tau_{k+1}) &  \mbox{ if } s = \tau_k,\; 1 \leq k\leq P-1\\
v^\prime(\tau_P) &  \mbox{ if }  s = \tau_P
\end{array}\right..
$$
Then the above discussion and Lemma \ref{GGG*G*} (a) yield
$$\sum\limits_{s,t \in \mathbb R} \langle f(s-t)h(s),h^\prime(t)\rangle =
\sum\limits_{s,t \in \mathbb R} \langle k(s,t)h(s),h(t)\rangle
=
\sum\limits_{s,t \in \mathbb R} \langle k(s,t)h^\prime(s),h^\prime(t)\rangle.
$$
This shows that, for each $h: \mathbb R \to \mathfrak H$ supported on a finite set and satisfying that \linebreak$\sum\limits_{s,t \in \mathbb R} \langle k(s,t)h(s),h(t)\rangle >0$, there exists $h^\prime: \mathbb R \to \mathfrak H$ with finite support such that 
$\sum\limits_{s,t \in \mathbb R} \langle k(s,t)h^\prime(s),h^\prime(t)\rangle >0$ and 
$$\frac{\left|\sum\limits_{s,t \in \mathbb R} \langle f(s-t)h(s),h^\prime(t)\rangle\right|}{\left\{\sum\limits_{s,t \in \mathbb R} \langle k(s,t)h(s),h(t)\rangle\right\}^{\frac{1}{2}}
\left\{\sum\limits_{s,t \in \mathbb R} \langle k(s,t)h^\prime(s),h^\prime(t)\rangle\right\}^{\frac{1}{2}}} =1.$$

It remains to verify that there exists a function $\rho: \mathbb R \to (0,\infty)$ such that
\begin{equation}\label{last}
\begin{array}{l}
\sum\limits_{s,t \in \mathbb R} \langle k(s+\xi,t+\xi)h(s),h(t)\rangle \leq \rho(\xi)\sum\limits_{s,t \in \mathbb R} \langle k(s,t)h(s),h(t)\rangle\\
\mbox{for all } h: \mathbb R \to \mathfrak H \mbox{ with finite support and all } \xi \in \mathbb R.
\end{array}
\end{equation}

Given $h: \mathbb R \to \mathfrak H$  with finite support, set 
$$h_\xi(s) := h(s-\xi)\quad (s \in \mathbb R)$$ 
and 
$$S(h,\xi) := \sum\limits_{s,t \in \mathbb R} \langle k(s+\xi,t+\xi)h(s),h(t)\rangle,$$
so that 
$$S(h,\xi) = \sum\limits_{s,t \in \mathbb R} \langle k(s,t)h_\xi(s),h_\xi(t)\rangle = S(h_\xi,0).$$
As before, 
$\sigma_Q < \cdots < \sigma_1 < \sigma_0 = 0 = \tau_0 < \tau_1 < \cdots < \tau_P$ are the points in 
$\mbox{supp}(h) \cup \{0\}$. 

We first consider the case ${\xi > 0}$.
 
The difference between the expressions for $S(h,\xi)$ and $S(h,0)$ depends on the set $\mbox{supp}(h) \cap [-\xi, 0)$.  \\
{\bf{Particular case:}} $ \sigma_1 \leq -\xi < 0$, i.e., $\mbox{supp}(h) \cap [-\xi, 0)$ is either $\{\sigma_1\}$ 
($\sigma_1 = -\xi$) or $\emptyset$ ($\sigma_1 < -\xi$). 
We obtain that
$$\begin{array}{lll}
S(h,\xi) &=& 
\sum\limits_{j=1}^Q {\int_{\sigma_j}^{\sigma_{j-1}}} \langle |G|T(u-\sigma_j)^*v(\sigma_j), 
T(u-\sigma_j)^*v(\sigma_j) \rangle\,du\\[.15cm] 
& & - {\int_0^\xi} \langle |G|T(u-\sigma_1-\xi)^*v(\sigma_1), T(u-\sigma_1-\xi)^*v(\sigma_1)\rangle \,du\\[.15cm]
& & + \|T(\sigma_1 -\xi)^*v(\sigma_1) + T(\xi)y(0)\|^2\\[.15cm]
& &+ {\int_0^\xi} \langle |G|T(u-\xi)y(0), T(u-\xi)y(0)\rangle \,du\\[.15cm]
& &+ \sum\limits_{k=1}^P {\int_{\tau_{k-1}}^{\tau_k}}\langle |G|T(\tau_k-u)v(\tau_k), T(\tau_k-u)v(\tau_k)\rangle \,du
\end{array}$$
where $y(0) = \sum\limits_{t\geq 0} T(t)h(t) = v(0) - T(-\sigma_1)^*v(\sigma_1)$. Since $|G|x = 2G^+x - Gx$
($x \in \mathcal D(G)$), from (\ref{DDT}) it follows that
$$\begin{array}{lll}
S(h,\xi) &=& S(h,0) -2\int_0^\xi \|(G^+)^{\frac{1}{2}}T(-t-\sigma_1)^*v(\sigma_1)\|^2\,dt\\[.15cm] 
& &+ 2\int_0^\xi\|(G^+)^{\frac{1}{2}}T(t)y(0)\|^2\,dt\\[.15cm]
&=& S(h,0) -2\int_0^\xi \|(G^+)^{\frac{1}{2}}T(-t-\sigma_1)^*v(\sigma_1)\|^2\,dt \\[.15cm]
& &+ 2\int_0^\xi\|(G^+)^{\frac{1}{2}}T(t)(v(0) - T(-\sigma_1)^*v(\sigma_1))\|^2\,dt\\[.15cm]
&\leq& S(h,0) -2\int_0^\xi \|(G^+)^{\frac{1}{2}}T(-t-\sigma_1)^*v(\sigma_1)\|^2\,dt \\[.15cm]
& &+ 4\int_0^\xi\|(G^+)^{\frac{1}{2}}T(t)v(0)\|^2\,dt \\[.15cm]
& &+ 4\int_0^\xi\|(G^+)^{\frac{1}{2}}T(t)T(-\sigma_1)^*v(\sigma_1)\|^2\,dt\\[.15cm]
&\leq& S(h,0) + 4\int_0^\xi\|(G^+)^{\frac{1}{2}}T(t)v(0)\|^2\,dt \\[.15cm]
& &+ 4\int_0^\xi\|(G^+)^{\frac{1}{2}}T(t)T(-\sigma_1)^*v(\sigma_1)\|^2\,dt .
\end{array}
$$
We apply Lemma \ref{GGG*G*} (c) and get
$$4\int_0^\xi\|(G^+)^{\frac{1}{2}}T(t)v(0)\|^2\,dt \leq 8\beta \int_0^\xi\|T(t)v(0)\|^2\,dt
\leq 8\beta \left(\int_0^\xi e^{2\beta t}\,dt\right)\|v(0)\|^2$$
and
$$\begin{array}{l}
4\int_0^\xi\|(G^+)^{\frac{1}{2}}T(t)T(-\sigma_1)^*v(\sigma_1)\|^2\,dt \\[.15cm] 
\quad = 4\int_0^\xi\|(G^+)^{\frac{1}{2}}T(t)T(t)^*T(-t-\sigma_1)^*v(\sigma_1)\|^2\,dt\\[.15cm]
\quad \leq 8\beta \int_0^\xi e^{2\beta t}\|T(t)^*T(-t-\sigma_1)^*v(\sigma_1)\|^2\,dt\\[.15cm]
\quad = 8\beta \int_0^\xi \int_0^t e^{2\beta t} \langle GT(s)^*T(-t-\sigma_1)^*v(\sigma_1), 
T(s)^*T(-t-\sigma_1)^*v(\sigma_1)\rangle\,ds\,dt\\[.15cm]
\quad\quad -8\beta \int_0^\xi e^{2\beta t}\| T(-t-\sigma_1)^*v(\sigma_1)\|^2\,dt\\[.15cm]
\quad \leq 8\beta \int_0^\xi \int_0^t e^{2\beta t} \langle GT(s)^*T(-t-\sigma_1)^*v(\sigma_1), 
T(s)^*T(-t-\sigma_1)^*v(\sigma_1)\rangle\,ds\,dt\\[.15cm]
\quad = 8\beta \int_0^\xi \int_0^t e^{2\beta t} \langle GT({s-t}-\sigma_1)^*v(\sigma_1), 
T(s-t-\sigma_1)^*v(\sigma_1)\rangle\,ds\,dt\\[.15cm]
\quad = 8\beta \int_0^\xi \int_0^x e^{2\beta x} \langle GT(-y-\sigma_1)^*v(\sigma_1), 
T(-y-\sigma_1)^*v(\sigma_1)\rangle\,dy\,dx\\[.25cm]
\quad \leq 8\beta \left( \int_0^\xi e^{2\beta x} \,dx\right)\left( \int_0^\xi \|CT(-y-\sigma_1)^*v(\sigma_1)\|^2\,dy\right)
\\[.15cm]
\quad = 8\beta \left( \int_0^\xi e^{2\beta x} \,dx\right)\left( \int_{-\xi}^0 \|CT(u-\sigma_1)^*v(\sigma_1)\|^2\,du\right).
\end{array}$$
It thereby follows that
$$
S(h,\xi) - S(h,0)  \leq  8\beta \left(\int_0^\xi e^{2\beta t}\,dt\right)
\left(\|v(0)\|^2 
+ \int_{-\xi}^0 \|CT(u-\sigma_1)^*v(\sigma_1)\|^2\,du\right).
$$
Since 
$$8\beta \int_0^\xi e^{2\beta t}\,dt \leq 8\beta \int_0^\xi e^{8\beta t}\,dt =  e^{8\beta\xi} - 1$$
and
$$\|v(0)\|^2 + \int_{-\xi}^0 \|CT(u-\sigma_1)^*v(\sigma_1)\|^2\,du \leq S(h,0)$$
it comes that
\begin{equation}\label{lastcondition}
S(h,\xi) \leq e^{8\beta\xi} S(h,0).
\end{equation}

We already mentioned that the cases to be considered are linked up with the points in $\mbox{supp}(h) \cap [-\xi, 0)$. Next we will see that all cases can be reduced to the particular one.

If $\sigma_2 \leq -\xi < \sigma_1$ then $S(h,\xi) = S(h, (\xi +\sigma_1) -\sigma_1) = S(h_{-\sigma_1}, \xi + \sigma_1)$ where $\sigma_2 - \sigma_1 \leq -(\xi + \sigma_1) < 0$, which is to say that 
$\mbox{supp}(h_{-\sigma_1}) \cap [-(\xi + \sigma_1), 0)$ is either $\{\sigma_2 - \sigma_1\}$ 
($\sigma_2 = -\xi$) or $\emptyset$ ($\sigma_2 < -\xi$). Therefore, $h_{-\sigma_1}$ and $\xi + \sigma_1$ in the sum 
$S(h_{-\sigma_1}, \xi + \sigma_1)$ are like in the particular case. Whence
$$S(h_{-\sigma_1}, \xi + \sigma_1) \leq
e^{8\beta(\xi + \sigma_1)}S(h_{-\sigma_1},0) = e^{8\beta(\xi + \sigma_1)}S(h, -\sigma_1).$$
Since
$h$ and $-\sigma_1$ in the sum $S(h, -\sigma_1)$ accord with the particular case as well, we get $S(h, -\sigma_1) \leq e^{-8\beta\sigma_1}
S(h,0)$, hence (\ref{lastcondition}). The case $\sigma_3 \leq -\xi < \sigma_2 $ can be carried over into the previous one and so on. 

The case $\xi <0$ can be handled in a similar way. We get that
$$S(h,\xi) \leq e^{-8\beta\xi} S(h,0).$$

Therefore (\ref{last}) holds with $\rho(\xi) := e^{8\beta|\xi|}$.

We have proved the following result.
\begin{thm}\label{first}
Let $\{T(t)\}$ be a strongly continuous semigroup of bounded linear operators on a Hilbert space $(\mathfrak H, \langle\cdot,\cdot\rangle)$. Let $A$ be the infinitesimal generator of $\{T(t)\}$. Suppose that there exists $\beta \geq 0$ such that $A-\beta$ is sectorial of angle $0<\theta<\frac{\pi}{2}$ and ${\mbox{\textnormal{Re}}}\langle Ax, x \rangle \leq \beta\|x\|^2$ for all $x \in \mathcal D(A)$. Then there exist a Kre\u{\i}n space $(\mathfrak K, \langle\cdot,\cdot\rangle_\mathfrak K)$ containing $(\mathfrak H, \langle\cdot,\cdot\rangle)$ as a regular subspace and a strongly continuous unitary group $\{U(t)\} \subseteq \mathfrak L(\mathfrak K)$ such that:
\begin{itemize}
\item[(a)] $T(t) = P_\mathfrak H U(t)|_\mathfrak H$ and $T(t)^* = P_\mathfrak H U(-t)|_\mathfrak H$ for all $t\geq 0$ (dilation property).
\item[(b)] $\bigvee \{U(t)\mathfrak H~:~t\in \mathbb R\}$ is dense in $\mathfrak K$  (minimality condition).
\end{itemize}
\end{thm}

It is all set for the result we announced in the introduction.
\begin{thm}\label{main}
Let $\{T(t)\}$ be a strongly continuous semigroup of bounded linear operators on a Hilbert space 
$(\mathfrak H, \langle\cdot,\cdot\rangle)$. Let $A$ be the infinitesimal generator of $\{T(t)\}$. Suppose that there exists $\beta \geq 0$ such that $A-\beta$ is sectorial of angle $0<\theta<\frac{\pi}{2}$ and ${\mbox{\textnormal{Re}}}\langle Ax, x \rangle_0 \leq \beta\|x\|_0^2$ for all $x \in \mathcal D(A)$, with $\langle\cdot,\cdot\rangle_0$ an equivalent scalar product on $\mathfrak H$. Then there exist a Kre\u{\i}n space $(\mathfrak K, \langle\cdot,\cdot\rangle_\mathfrak K)$ containing 
$(\mathfrak H, \langle\cdot,\cdot\rangle)$ as a regular subspace and a strongly continuous unitary group $\{U(t)\} \subseteq \mathfrak L(\mathfrak K)$ such that:
\begin{itemize}
\item[(a)] $T(t) = P_\mathfrak H U(t)|_\mathfrak H$ and $T(t)^* = P_\mathfrak H U(-t)|_\mathfrak H$ for all $t\geq 0$ .
\item[(b)] $\bigvee \{U(t)\mathfrak H~:~t\in \mathbb R\}$ is dense in $\mathfrak K$.
\end{itemize}
\end{thm}
\begin{proof}
By Theorem \ref{first} we already have a Kre\u{\i}n space 
$(\mathfrak K_0, \langle\cdot,\cdot\rangle_{\mathfrak K_0})$ containing 
$(\mathfrak H, \langle\cdot,\cdot\rangle_0)$ as a regular subspace and a strongly continuous unitary group $\{U_0(t)\} \subseteq \mathfrak L(\mathfrak K_0)$ such that $(\mathfrak K_0, \{U_0(t)\})$ is a minimal unitary dilation of $\{T(t)\} \subseteq \mathfrak L(\mathfrak H)$. 

We can choose a fundamental decomposition of $\mathfrak K_0$ such that $\mathfrak H \subseteq \mathfrak K_0^+$, say $\mathfrak K_0 = (\mathfrak H \oplus \mathfrak M^+) \oplus \mathfrak K_0^-$. Consider the linear space $\mathfrak K := \mathfrak K_0$ with the scalar product 
$$\langle h+ b, h^\prime + b^\prime\rangle_\mathfrak K := \langle h, h^\prime \rangle + 
\langle b,b^\prime\rangle_{\mathfrak K_0}\quad(h, h^\prime \in \mathfrak H, b,b^\prime \in \mathfrak M^+ \oplus \mathfrak K_0^-).$$
It is clear that $(\mathfrak K, \langle\cdot,\cdot\rangle_\mathfrak K)$ is a Kre\u{\i}n space containing $(\mathfrak H, \langle\cdot,\cdot\rangle)$ as a regular subspace for a fundamental decomposition can be chosen to be 
$$\mathfrak K=(\mathfrak H \oplus \mathfrak M^+) \oplus \mathfrak K_0^-.$$
The corresponding Hilbert space scalar product is
$$\langle h+ b, h^\prime + b^\prime\rangle_{|\mathfrak K|} := \langle h, h^\prime \rangle + 
\langle b,b^\prime\rangle_{|\mathfrak K_0|}\quad(h, h^\prime \in \mathfrak H, b,b^\prime \in \mathfrak M^+ \oplus \mathfrak K_0^-).$$

There exist $d, D >0$ such that $d\|h\|_0 \leq \|h\| \leq D\|h\|_0$ for all $h\in \mathfrak H$. Thus, there exists a  linear operator
$\Lambda$ on $\mathfrak H$ such that 
$$\langle h, h^\prime \rangle = \langle\Lambda h, h^\prime \rangle_0
\quad\mbox{for all } h, h^\prime \in \mathfrak H$$ 
and 
$$ d^2\|h\|_0^2 \leq \langle \Lambda h, h\rangle_0 \leq D^2\|h\|_0^2\quad\mbox{for all } h \in \mathfrak H.$$
Besides, $\Lambda$ is $\langle\cdot,\cdot\rangle$-positive (hence, bounded in the norm $\|\cdot\|$). Therefore $\Lambda$ is positive, bounded and boundedly invertible in both Hilbert spaces $(\mathfrak H, \langle\cdot, \cdot\rangle_0)$ and $(\mathfrak H, \langle\cdot,\cdot\rangle)$. $\Lambda$ can be extended to a linear operator $\widetilde\Lambda$ on all of $\mathfrak K = \mathfrak K_0$   
by putting $\widetilde\Lambda h:= \Lambda h$ ($h\in \mathfrak H$) and $\widetilde \Lambda b := b$ ($b \in \mathfrak M^+ \oplus \mathfrak K_0^-$). It follows that 
$$ \langle k, k^\prime \rangle _\mathfrak K = \langle \widetilde\Lambda k, k^\prime\rangle_{\mathfrak K_0}
=  \langle k, \widetilde\Lambda k^\prime\rangle_{\mathfrak K_0}\quad
\mbox{for all }k, k^\prime \in \mathfrak K = \mathfrak K_0$$
and 
$$(d^2 \wedge 1)\|k\|_{|\mathfrak K_0|}^2 \leq \langle\widetilde\Lambda k, k \rangle_{|\mathfrak K_0|}
\leq (D^2 \vee 1)\|k\|_{|\mathfrak K_0|}^2\quad\mbox{for all } k \in \mathfrak K_0.$$
Also, with respect to the decomposition $\mathfrak K = \mathfrak K_0 = \mathfrak H \oplus (\mathfrak M^+ \oplus \mathfrak K_0^-)$, 
$\widetilde\Lambda$ can be represented in the form
$$\widetilde \Lambda = \begin{bmatrix}
\Lambda & 0\\ 0 & 1
\end{bmatrix}=
\begin{bmatrix}
\Lambda^{\frac{1}{2}} & 0\\ 0 & 1
\end{bmatrix}
\begin{bmatrix}
\Lambda^{\frac{1}{2}} & 0\\ 0 & 1
\end{bmatrix}.$$ 
Thus, if $L :=  \begin{bmatrix}
\Lambda^{\frac{1}{2}} & 0\\ 0 & 1
\end{bmatrix}$ then $L$ is positive, bounded and boundedly invertible in both Kre\u{\i}n spaces $(\mathfrak K, \langle\cdot,\cdot\rangle_\mathfrak K)$ and $(\mathfrak K_0, \langle\cdot,\cdot\rangle_{\mathfrak K_0})$. 

Define $U(t) := L^{-1}U_0(t)L$ ($t \in \mathbb R$). It can be seen that $\{U(t\}\subseteq \mathfrak L(\mathfrak K)$ is a strongly continuous unitary group with the desired properties.  
 \end{proof}

%
%
%
%
%



\end{document}